\documentclass{amsart}
\usepackage{amsmath,amsthm,amsfonts,amssymb}
\usepackage{enumitem}
\usepackage{framed}
\usepackage{array}
\usepackage{nicefrac}
\usepackage[english]{babel}
\usepackage{color,xcolor,graphicx}
\usepackage{comment}
\usepackage{ytableau}
\usepackage{tikz}
\usetikzlibrary{cd}
\usepackage{hyperref}
\usepackage[capitalize]{cleveref}
\theoremstyle{plain}
\newtheorem{lemma}[subsection]{Lemma}
\newtheorem{theorem}[subsection]{Theorem}
\newtheorem{proposition}[subsection]{Proposition}
\newtheorem{corollary}[subsection]{Corollary}
\newtheorem*{theorem*}{Theorem}

\newtheorem{mainTheorem}{Theorem}

\crefname{mainTheorem}{Theorem}{Theorems}

\theoremstyle{definition}

\theoremstyle{remark}

\newtheorem*{remark*}{Remark}

\numberwithin{equation}{subsection} 
\numberwithin{figure}{section} 
\numberwithin{table}{section} 
\numberwithin{subsection}{section} 

\makeatletter
\let\c@figure\c@subsection
\let\c@table\c@subsection
\makeatother

\newcommand\iso{\ensuremath{\xrightarrow{{}_\sim}}}
\newcommand\set[2]{\ensuremath{\left\{#1\mid#2\right\}}}

\DeclareMathOperator{\Hom}{Hom}
\DeclareMathOperator{\Tab}{Tab}
\DeclareMathOperator{\gl}{\ensuremath{\mathfrak{gl}}}

\DeclareMathOperator{\rk}{rk\,}
\newcommand\T{\ensuremath{\mathsf T}}
\def\P{\ensuremath{\mathsf P}}
\newcommand\Q{\ensuremath{\mathsf Q}}
\newcommand\Fl{\ensuremath{\mathcal Fl}}
\newcommand\Ni{\ensuremath{\mathcal N}}
\newcommand\fibre[3]{\ensuremath{S_{\underline{#1}}^{\mathsf{#2}}(#3)}}
\newcommand\fl[2]{\ensuremath{\mathcal Fl(\underline #1, #2)}}
\ytableausetup{smalltableaux}
\newcommand\minus{\scalebox{0.75}{\ytableaushort -}}
\newcommand\plus{\scalebox{0.75}{\ytableaushort +}}

\newcommand\pr{\ensuremath{\operatorname{pr}}}
\newcommand\into\hookrightarrow
\newcommand\ev[1]{\ensuremath{\operatorname{ev}(#1)}}
\newcommand\Par{\ensuremath{\mathcal Par}}
\newcommand\sh[1]{\operatorname{sh}(#1)}

\newcommand\rect[2]{\ensuremath{\operatorname{Rect}(\mathsf #1, #2)}}

\title{A Robinson-Schensted Correspondence for Partial Permutations}
\author{Rahul Singh}
\address{Department of Mathematics, Virginia Tech, 460 McBryde Hall, 225 Stanger St., Blacksburg VA 24061}
\email{rahul.sharpeye@gmail.com}

\begin{document}

\maketitle 
\begin{abstract}
We study the Steinberg variety associated to matrix Schubert varieties,
and develop a Robinson-Schensted type correspondence, 
$\tau\leftrightarrow(\Lambda,\Q,\P)$.
Here $\tau$ is a partial permutation of size $p\times q$,
$\Lambda$ an admissible signed Young diagram of size $p+q$,
and \P\ (resp. \Q) a standard Young tableau of size $p$ (resp. $q$) whose shape is determined by $\Lambda$.
By embedding the matrix Schubert variety into a Schubert variety,
we find a close relationship between the combinatorics of the classical Robinson-Schensted-Knuth correspondence and our bijection.
We also show that an involution $(\Lambda,\Q,\P)\mapsto(\Lambda^\vee,\P,\Q)$ corresponds to projective duality on matrix Schubert varieties.
\end{abstract}

\section{Introduction}

The classical Robinson-Schensted correspondence,
see \cite{MR1507943,MR1464693},
associates to each permutation on $n$ letters,
a triple $(\lambda,\Q,\P)$, with $\lambda$ a partition of $n$,
and \Q\ and \P\ Young tableaux of shape $\lambda$.
In \cite{MR929778},
Steinberg presented a geometric interpretation of the Robinson-Schensted correspondence,
by showing that both sides of the correspondence count the irreducible components of the Steinberg variety.
This work has seen many generalizations,
see \cite{MR2128023,MR2966826,MR2979579,nishiyama} among others.
We continue in this tradition,
developing a bijection for partial permutations
by studying the Steinberg variety associated to matrix Schubert varieties.

Let $\Fl(V)$ be the flag variety of a vector space $V$.
Given a nilpotent map $x\in\gl(V)$, the variety
\begin{align*}
\set{(x,F_\bullet)\in\gl(V)\times\Fl(V)}{xF_i\subset F_{i-1}\,\forall\,i}
\end{align*}
is called the \emph{Springer fibre} over $x$.
Following \cite{MR672610}, 
the irreducible components of the Springer fibre
are indexed by the standard Young tableaux of shape $J(x)$,
the Jordan type of $x$.

Let $V_p$ (resp. $V_q$) be a $p$ (resp. $q$)-dimensional vector space, and
\begin{align*}
K=GL(V_q)\times GL(V_p), && C=\Fl(V_q)\times\Fl(V_p)\times\Hom(V_q,V_p).
\end{align*}
The $K$-orbits in $C$ are called \emph{matrix Schubert varieties};
they are indexed by partial permutations of size $p\times q$,
see \cite{MR1154177}.
In a very general setting, see \cref{sec:conormal},
the \emph{irreducible components of the corresponding Steinberg variety} $Z$
are known to be \emph{in bijection with the $K$-orbits},
i.e., the matrix Schubert varieties.

Let $\mathcal O$ be the \emph{nilpotent cone of the} $\widehat A_2$ \emph{quiver}.
The $K$-orbits of $\mathcal O$ are indexed by signed Young diagrams,
cf. \cite{MR666395,MR2941520}, see also \cref{SYD}.
We construct a $K$-equivariant proper map, $\pr:Z\to\mathcal O$,
closely related to the moment map for the $K$-action on $C$.
The fibre at each point of this map is a \emph{product of Springer fibres}.
%cf. \cite{MR672610,MR929778}, see \cref{specialDegen}.
In particular, the irreducible components of this fibre are indexed by certain pairs of standard Young tableaux.

This suggests an alternate characterization of the irreducible components of $Z$.
For every $Z'\in\operatorname{Irr}(Z)$,
there is a signed Young diagram $\Lambda$ such that
$\pr\left(Z'\right)=\overline{\mathcal O\left(\Lambda\right)}$;
let $(\Q,\P)$ be the pair of standard Young tableaux indexing the generic fibre of $\pr:Z'\to\overline{\mathcal O(\Lambda)}$.
The triple $(\Lambda,\Q,\P)$ identifies the irreducible component $Z'$.

Our first result, \Cref{main:bijection}, states that the above procedure does yield a bijection, with one caveat:
only a subset of the signed Young diagrams appear as $\Lambda$;
we call these the \emph{admissible signed Young diagrams}.
We give geometric and combinatorial characterizations of admissible singed Young diagrams in \cref{defn:admissible,main:admissible}.

\begin{mainTheorem}
Let $\tau$ be a partial permutation of size $p\times q$,
and let $(F_\bullet,G_\bullet,x,y)$ be a generic point in the matrix Schubert variety $C(\tau)$.
Let $\pr(C(\tau))=\mathcal O(\Lambda)$, and set $\Q=\Tab(yx,F_\bullet)$, $\P=\Tab(xy,G_\bullet)$.
We have a bijection,
\begin{equation*}
\mathcal{PP}(p,q)=\bigsqcup\limits_{\Lambda\in ASYD(q,p)}SYT(\Lambda^+)\times SYT(\Lambda^-),
%{(\Lambda,\Q,\P)}{\Lambda\in ASYD(q,p),\,\Q\in SYT(\Lambda^+),\,\P\in SYT(\Lambda^-)}.
\end{equation*}
given by $\tau\mapsto (\Lambda,\Q,\P)$.
\end{mainTheorem}
The flag $F_\bullet$ (resp. $G_\bullet$) lives in the Springer fibre above the nilpotent matrix $yx$ (resp. $xy$),
and \Q\ (resp. \P) is the standard Young tableau relating the $F_\bullet$ (resp. $G_\bullet$) with $yx$ (resp $xy$), see \cref{springerFibre}.

To unravel the combinatorics of the bijection in \cref{main:bijection},
we \emph{embed the matrix Schubert variety $C(\tau)$ into a Schubert variety $Y(\widehat\tau)$},
see \cref{sub:flags,sub:schubert},
via (a simple variation of) the map,
\begin{equation*}
x\mapsto\begin{pmatrix}I_q&0\\x&I_p\end{pmatrix},
\end{equation*}
see \cref{mapEmbed}.
The key idea is that the conormal variety of $C(\tau)$ is closely related to the conormal variety of $Y(\widehat\tau)$,
see \cref{relateRSK}.
This allows us to leverage the geometric interpretation of the RSK correspondence (cf. \cite{MR929778,MR2979579});
combining this with the \emph{geometric interpretations of various tableau algorithms} developed by van Leeuwen, cf. \cite{MR1739585},
we obtain, in \cref{main:PPtoTriple,main:tripleToPP},
combinatorial descriptions of the bijection in \cref{main:bijection}.
%our other results,

\begin{mainTheorem}
\label{main:PPtoTriple}
Consider $\tau\in\mathcal{PP}(p,q)$.
% and let $\widehat\tau$ be as in \cref{tauHat}.
Let $(\widehat\Q,\widehat\P)$ be the tableaux pair corresponding to $\widehat\tau$ via the Robinson-Schensted-Knuth correspondence,
and set
\begin{align*}
\lambda=\sh{\widehat\Q},&& \Q=\widehat\Q(q),&& \P=\operatorname{Rect}(\widehat\P,q).
\end{align*}
Let $\Lambda$ be the signed Young diagram given by the following rules:
\begin{enumerate}
\item
$\sh\Lambda=\sh\Q+\sh\P$.
\item
The $i^{th}$-row of $\Lambda$ starts with \plus\ if and only if $\sh\Q(i)>\lambda(i)$.
\end{enumerate}
The bijection of \cref{main:bijection} is made explicit by the relation: $\tau\leftrightarrow(\Lambda,\Q,\P)$.
\end{mainTheorem}
See \cref{tauHat} for the definition of $\widehat\tau$,
\cref{rst} for the definition of $\widehat\Q(q)$,
\cref{rectification} for a definition of $\operatorname{Rect}(\widehat\P,q)$,
and \cref{steinberg} for a discussion of the Robinson-Schensted-Knuth correspondence.

\begin{mainTheorem}
\label{main:tripleToPP}
Consider $\Lambda\in ASYD(q,p)$, $\Q\in SYD(\Lambda^+)$, and $\P\in SYD(\Lambda^-)$,
where $\Lambda^+$ and $\Lambda^-$ be as in \cref{LambdaPM}.
Let $\lambda$ be as in \cref{nuLambda}, let
\begin{align*}
\widehat\Q=(\Q;\lambda),&& \ev{\widehat\P}=\ev{(\ev\P;\lambda)},&&\widehat\tau\overset{RSK}\longleftrightarrow(\widehat\Q,\widehat\P).
\end{align*}
Then $\tau$ is the south-west sub-matrix of $\widehat\tau$ of size $p\times q$.
\end{mainTheorem}
The operator $\mathrm{ev}$ is tableau evacuation, also known as the Sch\"utzenberger involution, see \cref{evacuation}.

Given a signed Young diagram $\Lambda$, let $\Lambda^\vee$ be the diagram obtained by switching the labelling of the boxes.
In \cref{duality}, we show that the involution $(\Lambda,\Q,\P)\mapsto(\Lambda^\vee,\P,\Q)$ corresponds to projective duality for matrix Schubert fibres.

The questions tackled here are also discussed in \cite{nishiyama},
where a Robinson-Schensted correspondence is developed for partial permutations of size $n\times n$.
However, both the bijection and the algorithms presented in \cite[Thm. 7.6]{nishiyama} are different from the ones here.
Another key difference is that we work with Schubert varieties in a double partial flag variety $\Fl(\underline a,V)\times\Fl(\underline b,V)$,
as contrasted with the Grassmannian subvarieties used in \cite{nishiyama}.
This allows us to factor much of the combinatorics through the combinatorics of the classical Robinson-Schensted-Knuth correspondence.

{\em Acknowledgements.}
The author would like to thank K. Nishiyama for some very illuminating discussions.

\setcounter{mainTheorem}{0}
\section{The Steinberg Variety}
\label{sec:conormal}
We work over $\mathbb C$.
In this section, we recall some standard results on the Steinberg variety of a $G$-variety.
For a detailed discussion, the reader may consult \cite{MR1433132}.

\subsection{The Conormal Bundle}
\label{sub:con}
Let $X$ be a smooth variety, $Y$ a smooth (not necessarily closed) subvariety of $X$,
and $TX$ and $TY$ the corresponding tangent bundles.
The conormal bundle of $Y$ in $X$ is a vector bundle,
$%\[
	T^*_XY\to Y
$, %\]
whose fibre at a point $p\in Y$ is the annihilator of the tangent subspace $T_pY$ in $T^*_pX$, i.e.,
\[
	(T^*_XY)_p=\set{x\in T^*_pX}{x(v)=0,\forall v\in T_pY}.
\]

%Now suppose $Y$ is a closed, but not necessarily smooth, subvariety of $X$. 
%Let $Y^{sm}$ denote the smooth locus of $Y$.
%The closure (in $T^*X$) of the conormal bundle $T^*_XY^{sm}$ is called the \emph{conormal variety} $T^*_XY$ of $Y$ in $X$.
%Restricting the natural projection $T^*X\to X$ to the conormal variety induces a \emph{structure map}, $T^*_XY\to Y$.
%The conormal variety $T^*_XY$ is not, in general, a vector bundle over $Y$.
%
%\begin{lemma}
%\label{lem:denseOpen}
%Suppose $Y^\circ$ is a dense open subset of $Y^{sm}$.
%The closure of $T^*_XY^\circ$ in $T^*X$ is precisely the conormal variety $T^*_XY$.
%\end{lemma}
%\begin{proof}
%%[\cite{comineqn}]
%Since $Y^\circ$ is open in \yr, we have $T_pY^\circ=T_p\yr$, and hence $(T^*_XY^\circ)_p=(T^*_X\yr)_p$, for all $p\in Y^\circ$. 
%It follows that $T^*_X{Y^\circ}$, being the pull-back of the vector bundle $T^*_X\yr$ along the open inclusion $Y^\circ\hookrightarrow\yr$, is dense in $T^*_XY^\circ$. 
%Consequently, we have $\overline{T^*_XY^\circ}=\overline{T^*_X\yr}=T^*_XY$.
%\end{proof}
%
\subsection{Group Actions with Finitely Many Orbits}
\label{finiteOrbits}
Let $G$ be a reductive group acting on a (possibly singular) algebraic variety $X$ with finitely many orbits.
We denote by $\nicefrac XG$ the set of $G$-orbits in $X$, and write 
\begin{align*}
X=\bigsqcup\limits_{\lambda\in\nicefrac XG}X(\lambda),
\end{align*}
for the $G$-orbit decomposition of $X$.
The orbit closure inclusion relation induces a partial order on $\nicefrac XG$, namely 
$
\lambda\preceq\nu\iff\overline{X(\lambda)}\subset\overline{X(\nu)}.
$
%
%Following \cref{lem:denseOpen}, the conormal variety $T^*_X\overline{X(\lambda)}$ of $\overline{X(\lambda)}$ is precisely the closure $\overline{T^*_XX(\lambda)}$ of the conormal bundle $T^*_XX(\lambda)$.
%It is $G$-stable, and Lagrangian with respect to the symplectic structure on $T^*X$.

\subsection{The Steinberg Variety}
\label{kashiwara}
The $G$-action on $X$ induces a $G$-action on the cotangent bundle $T^*X$;
The symplectic structure on $T^*X$ admits a $G$-equivariant proper map, $\mu_X:T^*X\to\mathfrak g$, called the \emph{moment map}.
The zero fibre of the moment map, $\mu_X^{-1}(0)$, is called the \emph{Steinberg variety} $Z_X$ of $X$.
%\begin{lemma}[cf. \cite{MR1433132}]
Its irreducible components are precisely
the closures of the conormal bundles of the $G$-orbits $X(\lambda)\subset X$.
We write this concisely as
\begin{align*}
\operatorname{Irr}(Z_X)=\set{\overline{T^*_XX(\lambda)}}{\lambda\in\nicefrac XG}.
\end{align*}
%\end{lemma}

\section{The Robinson-Schensted-Knuth Correspondence}
\label{sec:rsk}
Let $V$ be a $N$-dimensional vector space, and let $G=GL(V)$.
In this section, we recall some results on the nilpotent cone, flag varieties, and Springer fibres associated to $G$.
We also present the Robinson-Schensted-Knuth correspondence from the geometric point of view,
and some results relating algorithms on Young tableaux with certain geometric constructions.

\subsection{Partitions}
\label{sec:dominance}
A partition of $N$ is a weakly decreasing sequence
\begin{equation*}
\lambda=(\lambda(1),\cdots,\lambda(k))
\end{equation*}
of positive integers satisfying $\lambda(1)+\cdots+\lambda(k)=N$.
We call 
%%$\lambda(i)$ the \emph{$i^{th}$-part} of $\lambda$, and 
$|\lambda|=N$ the size of $\lambda$.
%We say that $\lambda$ is a partition of $|\lambda|$,
We denote by $\mathcal Par(N)$ the set of all partitions of size $N$.
%\subsection{Young Diagrams}
The \emph{Young diagram} of $\lambda$ is a collection of boxes,
arranged in left-justified rows, with the $i^{th}$-row containing precisely $\lambda(i)$ boxes.
\begin{figure}[h]
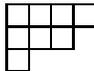

%\centering
\begin{align*}
\ydiagram{4,3,1} %&& \ydiagram{3,2,2,1}  && \ydiagram{7,5,3,1}\\
%\lambda=(4,3,1) %&& \lambda^T=(3,2,2,1) && \lambda+\lambda^T=(7,5,3,1)
\end{align*}
\caption{The Young Diagram of $\lambda=(4,3,1)$.}
\end{figure}

%The \emph{transpose} $\lambda^T$ of $\lambda$ is the partition obtained by flipping the diagram of $\lambda$ across the diagonal.
%Observe that transposition is an anti-automorphism with respect to the dominance order, i.e.,
%$
%\lambda\preceq\nu\implies\nu^T\preceq\lambda^T.
%$
%\subsection{Binary Operations on Partitions}

\subsection{The Nilpotent Cone} 
Let $\mathcal N$ be the nilpotent cone in $\gl(V)$, i.e,
\begin{equation*}
\mathcal N=\set{x\in\gl(V)}{x^N=0},
\end{equation*}
and let $\Ni(\lambda)\subset\Ni$ be the subvariety of nilpotent matrices whose Jordan type is $\lambda$.
Following \cite[V(2.9)]{MR553598}, we have, 
\begin{equation}
\label{macMagic}
\dim \mathcal N(\lambda)=|\lambda|\left(|\lambda|+1\right)-2\sum i\lambda(i).
\end{equation}
The $G$-orbit decomposition of \Ni\ is precisely, 
\begin{equation*}
\Ni=\bigsqcup\limits_{\lambda\in\Par(N)}\Ni(\lambda).
\end{equation*}
%In other words, we have $\nicefrac\Ni G=\Par(N)$.
Further, the closure inclusion order on $\nicefrac\Ni G=\Par(N)$ is precisely the so-called \emph{dominance order} $\preceq$.
Recall that for $\lambda,\nu\in\Par(N)$, we have% $\lambda\preceq\nu$ if and only if
\begin{align*}
\lambda\preceq\nu&&\iff&& \lambda(1)+\cdots+\lambda(i)\leq\nu(1)+\cdots+\nu(i),\qquad\forall\,i.
\end{align*}

\subsection{Sum of Partitions}
\label{sumOfPartitions}
%We define a pair of commutative and associative binary operations $+$ and $\sqcup$ on $\Par$.
Given $\nu\in\Par(q)$ and $\lambda\in\Par(p)$, we define% $\nu+\lambda,\nu\sqcup\lambda\in\Par(q+p)$ via,
\begin{align*}
%\nonumber
\nu+\lambda & =(\nu(1)+\lambda(1),\nu(2)+\lambda(2),\cdots)\in\Par(q+p).
\end{align*}

Observe that the operation $+$ of \cref{sumOfPartitions} respects the dominance order,
i.e., for any $\lambda,\nu\in\Par(N)$, and any $\mu\in\Par$,
%i.e., given $\mu\in\Par$, and $\lambda,\nu\in\Par(N)$ satisfying $\lambda\preceq\nu$,
we have $\lambda\preceq\nu$ if and only if $\lambda+\mu\preceq\nu+\mu$.

\subsection{The Young Lattice}
\label{YoungLattice}
Given partitions $\lambda$ and $\nu$, we say $\lambda\leq\nu$ if the Young diagram of $\lambda$ is contained in the Young diagram of $\nu$. %, i.e., if the $\lambda(i)\leq\mu_i$ for all $i$.
The set of all partitions of all positive integers forms a lattice, called the \emph{Young lattice}, under the partial order $\leq$.
The Young lattice is graded by size, i.e.,
\begin{equation*}
\lambda<\nu\implies|\lambda|<|\nu|.
\end{equation*}
It is clear that distinct partitions of the same size are mutually incomparable in the Young lattice.

\subsection{Column Strips}
Consider partitions $\lambda$, $\nu$ with $\lambda\leq\nu$.
The set-theoretic difference $\nicefrac\nu\lambda$ is called a \emph{skew-diagram},  see \cite{MR1464693}.
A \emph{column strip} is a skew-diagram in which every row contains at-most $1$ box.

\begin{figure}[ht]
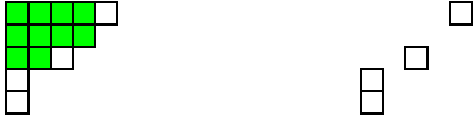

\begin{align*}
\ytableaushort[*(green)]{\ \ \ \ {*(white)\ },\ \ \ \ ,\ \ {*(white)\ },{*(white)\ },{*(white)\ }}&& \ytableaushort{\none\none\none\none\ ,\none\none\none\none,\none\none\ ,\ ,\ } 
\end{align*}
\caption{A pair of partitions $\lambda$ (green boxes only) and $\nu$ (all boxes). 
The skew-diagram $\nicefrac\nu\lambda$ is a Column Strip.}
\end{figure}

\subsection{Row-Standard Tableaux}
\label{rst}
%A \emph{composition} of $N$ is a sequence of positive integers $(a_1,\cdots)$ whose sum is $N$.
A \emph{composition} of $N$ is a sequence $\underline a=(a_1,\cdots,a_n)$ of positive integers satisfying $a_1+\cdots+a_n=N$. 

A \emph{row-standard tableau} $\mathsf T$ of content $\underline a$ is a sequence of partitions,
\begin{equation*}
\mathsf T=(\mathsf T(1);\cdots;T(n)),
\end{equation*}
satisfying for each $i$, the following conditions:
\begin{enumerate}
\item
$|\mathsf T(i)|=a_1+\cdots+a_i$.
\item 
$\T(i)<\T(i+1)$ in the Young lattice.
\item
The skew-diagram $\nicefrac{\mathsf T(i)}{\mathsf T(i-1)}$ is a column-strip.
\end{enumerate}

Equivalently, $\mathsf T$ is a filling of a Young diagram, which is weakly increasing in each column and strictly increasing in each row, with the integer $i$ occurring $a_i$ times, see \cite{MR1464693}.
We denote by $SYT(\lambda,\underline a)$ the set of row-standard Young tableaux of shape $\lambda$ and content $\underline a$.

\emph{
While column-standard tableaux abound in the literature,
row-standard tableaux are more natural for our purposes,
as will become clear in \cref{springerFibre,steinberg}; see also \cite{MR2979579}.
All the tableaux we encounter in this paper will be row-standard.
Accordingly, we drop the adjective row-standard, and simply call them tableaux.
}

\subsection{Partial Flag Varieties}
\label{sub:flags}
\label{rankMatrix}
Let $\underline a=(a_1,\cdots,a_n)$ be a \emph{composition} of $N$; we denote by $\fl aV $ the variety of partial flags in $V$ of shape $\underline a$, i.e.,
\begin{equation*}
\fl aV =\set{F_\bullet=(F_1\subset \cdots\subset F_n\subset V)}{\dim\nicefrac{F_i}{F_{i-1}}=a_i}.
\end{equation*}
Note that $G=GL(V)$ acts transitively on $\fl aV$.
Following \cite{MR0263830}, we have,
\begin{equation*}
T^*\fl aV =\set{(F_\bullet,x)\in\fl aV \times\gl(V)}{xF_i\subset F_{i-1},\,\forall\,1\leq i\leq n}.
\end{equation*}
The moment map for the $G$-action on \fl aV\ is given by,
\begin{align}
\label{springerMap}
&&\mu_{\fl aV}:T^*\fl aV \to\gl(V),&& (F_\bullet,x)\mapsto x.&&
\end{align}
In particular, the image of the moment map is contained in $\Ni$.
%the \emph{nilpotent cone} $\mathcal N$ of $G$, i.e., the variety of nilpotent elements in $\gl(V)$.

\subsection{The Springer Fibre}
\label{springerFibre}
Given $x\in\mathcal N$,
let $S_{\underline a}(x)=\mu_{\fl aV}^{-1}(x)$, i.e.,
\begin{equation*}
%S_{\underline a}(x)=\mu_{\fl aV}^{-1}(x)=\set{F_\bullet\in\fl aV}{(F_\bullet,x)\in T^*\fl aV}.
S_{\underline a}(x)=\set{F_\bullet\in\fl aV}{xF_i\subset F_{i-1}\,\forall\,1\leq i\leq n}.
\end{equation*}
We call $S_{\underline a}(x)$ the \emph{Springer fibre} over $x$.

Now consider $(x,F_\bullet)\in T^*\fl aV$.
For each $i$, we have $xF_i\subset F_{i-1}$, and hence a nilpotent map 
$x|F_i\in\gl(F_i)$ obtained by restricting $x$ to $F_i$.
Let $J(x|F_i)$ denote the Jordan type of $x|F_i$, and let
$\Tab(x,F_\bullet)$ denote the tableau. 
\begin{equation*}
\Tab(x,F_\bullet)= (J(x|F_1); J(x|F_2);\cdots; J(x|F_n)).
\end{equation*}
It is clear that the content of $Tab(x,F_\bullet)$ is $\underline a$.
We have a decomposition,
\begin{align*}
&&\fibre a{}x=\bigsqcup\fibre aTx, && \mathsf T\in SYT(J(x),\underline a),
\end{align*}
where the subvariety $S^{\mathsf T}_{\underline a}(x)\subset S_{\underline a}(x)$ is given by, 
\begin{align*}
\fibre aTx=\set{F_\bullet\in S_{\underline a}(x)}{\Tab(x,F_\bullet)=\mathsf T}.%&&\mathsf T\in SYT(J(x),\underline a).
\end{align*}
Following \cite{MR672610}, we have,
\begin{equation*}
\operatorname{Irr}\left(S_{\underline a}(x)\right)=\set{\overline{\fibre aTx}}{\mathsf T\in SYT(J(x),\underline a)}.
\end{equation*}

Recall that a variety is said to be pure-dimensional if all of its components have the same dimension.
Following \cite{MR672610}, the variety $S_{\underline a}(x)$ is pure-dimensional, with the dimension given by,
\begin{equation}
\label{dimSpringerFibre}
2\dim\fibre aTx=2\dim\fl aV-\dim\mathcal N(J(x)).
\end{equation}

\subsection{Schubert Varieties}
\label{sub:schubert}
Let $\underline a=(a_1,\cdots,a_n)$ and $\underline b=(b_1,\cdots,b_m)$ be a pair of compositions of $N$.
The group $G$ acts diagonally on $Y=\fl aV \times\Fl(\underline b,V)$ with finitely many orbits, the so-called \emph{Schubert cells}.

%The Schubert cells in $\fl aV\times\fl bV$ are indexed by $m\times n$ matrices $\sigma$ satisfying the following conditions:
The Schubert cells $Y(\sigma)\subset Y$ are indexed by $m\times n$ matrices $\sigma$ satisfying the following conditions:
\begin{itemize}
\item Each entry of $\sigma$ is a non-negative integer.
\item The sum of the entries in the $i^{th}$-row is $b_i$.
\item The sum of the entries in the $i^{th}$-column is $a_i$.
\end{itemize}
The Schubert cell $Y(\sigma)$ corresponding to $\sigma$ is then precisely,
\begin{equation*}
Y(\sigma)=\set{\left(E_\bullet,F_\bullet\right)\in Y}{\rk(E_i\cap F_j)= \sum\limits_{k=1}^i\sum\limits_{l=1}^j\sigma_{kl}}.
\end{equation*}
The closure of a Schubert cell is called a \emph{Schubert variety}.
We have,
%The corresponding orbit closure is given by
\begin{equation*}
\overline{Y(\sigma)}=\set{\left(E_\bullet,F_\bullet\right)\in Y}{\rk\left(E_i\cap F_j\right)\geq \sum\limits_{k=1}^i\sum\limits_{l=1}^j\sigma_{kl}}.
\end{equation*}

\subsection{The Steinberg Variety of the Double Flag Variety}
\label{steinDouble}
The moment map,
\begin{equation*}
\mu_Y:T^*\fl aV\times T^*\fl bV\to\gl(V),
\end{equation*}
for the $G$ action on $Y=\fl aV\times\fl bV$ is given by
\begin{equation*}
((F_\bullet,x),(E_\bullet,y))\mapsto x+y.
\end{equation*}
Consequently, the corresponding Steinberg variety $Z_{Y}$ is precisely,
\begin{align*}
Z_{Y} & =\set{(F_\bullet,x,E_\bullet,y)\in T^*\fl aV\times T^*\fl bV}{y=-x},\\
                 & =\set{(F_\bullet,E_\bullet,x)\in Y\times\Ni}{xF_i\subset F_{i-1},\,xE_i\subset E_{i-1},\,\forall i}.
\end{align*}
For $\sigma\in\nicefrac YG$, let $Z_{Y}(\sigma)=T^*_{Y}Y(\sigma)$.
Following \cref{kashiwara}, we have,
\begin{equation*}
\operatorname{Irr}(Z_Y)=\set{\overline{Z_Y(\sigma)}}{\sigma\in\nicefrac YG}.
\end{equation*}

%\begin{proposition}
\subsection{Geometry of the Robinson-Schensted-Knuth correspondence}
\label{steinberg}
Following \cite{MR272654,MR929778,MR2979579}, we have a bijection,
%\nicefrac YG\iso\bigsqcup\limits_{\lambda\vdash N} SYT(\lambda,\underline a)\times SYT(\lambda,\underline b), &&
\begin{align*}
\nicefrac YG & \iso\bigsqcup\limits_{\lambda\vdash N} SYT(\lambda,\underline a)\times SYT(\lambda,\underline b), \\
\sigma       & \mapsto(\Tab(x,E_\bullet),\Tab(x,F_\bullet)),
\end{align*}
where $x\in\mathcal N$ is chosen to be generic for the property $(E_\bullet,F_\bullet,x)\in T^*_YY(\sigma)$.
%Consider $(F_\bullet,G_\bullet)\in Y(\sigma)$, and let $x\in\mathcal N$ be generic for the property $(F_\bullet,G_\bullet,x)\in T^*_YY(\sigma)$.
%\begin{equation*}
%\sigma\mapsto(\Tab(x,F_\bullet),\Tab(x,G_\bullet)),
%\end{equation*}
This bijection is a simple variation of the correspondence discussed in \cite{MR1464693},
see \cite{MR2979579} for details.
%Robinson-Schensted-Knuth correspondence, see \cite{MR1464693}.
%\framebox{Move this citation to the introduction}
%\end{proposition}

Suppose $\sigma\in\nicefrac YG$ corresponds to the tableaux pair $(\P,\Q)$ via the above correspondence.
We will denote
\begin{align}
\label{ZYcirc}
Z^\circ_Y(\sigma)=\set{(E_\bullet,F_\bullet,x)}{\P=\Tab(x,E_\bullet),\Q=\Tab(x,F_\bullet)}.
\end{align}
Following \cite{MR929778}, $Z^\circ_Y(\sigma)$ is dense open subset of $Z_Y(\sigma)$.

\subsection{Rectification}
\label{rectification}
Given a tableau $\mathsf T$, let $F_{\bullet}$ be a generic point in $\fibre aTx$, see \cref{springerFibre}.
Fix $0\leq i\leq n$, and consider the partial flag,
\begin{equation*}
\nicefrac{F_\bullet}{F_i}=0\subset\nicefrac{F_{i+1}}{F_i}\subset\cdots\subset\nicefrac{F_n}{F_i}.
\end{equation*}
%in the vector space $\nicefrac{F_n}{F_i}$.
Since we have $xF_i\subset F_{i-1}\subset F_i$, we obtain an induced nilpotent map, 
\begin{align*}
&&x:\nicefrac{F_n}{F_i}\to\nicefrac{F_n}{F_i}, && \text{satisfying} && x(\nicefrac{F_j}{F_i})\subset\nicefrac{F_{j-1}}{F_i},\ \forall\, i\leq j\leq n.&&
\end{align*}
%Consequently, we obtain an induced tableau,
This yields the tableau,
\begin{align*}
\rect Ti:=\Tab(x|\nicefrac{F_n}{F_i},\nicefrac{F_\bullet}{F_i}).
\end{align*}
Following \cite{MR1739585}, \rect Ti\ is precisely the \emph{rectification} of the skew-tableau $\nicefrac{\mathsf T}{\mathsf T(i)}$ as described in \cite{MR1464693}.

\subsection{Evacuation}
\label{evacuation}
%Let $\mathsf T$ be a (row-standard) tableau with content $\underline a$.
%Given a tableau $\mathsf T$, let $F_{\bullet}$ be a generic point in $\fibre aTx$.
%Since $xF_i\subset F_i$, for all $i$, we have induced maps, $x|\nicefrac{F_n}{F_{i-1}}$.
Let $F_\bullet$ and $x$ be as in \cref{rectification}.
The sequence,
\begin{equation*}
(J(x|\nicefrac{F_n}{F_{n-1}}); J(x|\nicefrac{F_n}{F_{n-2}});\cdots; J(x|\nicefrac{F_n}{F_0})),
\end{equation*}
is a tableau of content $(a_n,a_{n-1},\cdots,a_1)$.
We denote this tableau by \ev\T.
Following \cite{MR1739585}, \ev\T\ is precisely the evacuation (also known as the Sch\"utzenberger involution) of \T, see \cite[A3.8]{MR2133266} for more details, including an algorithm to compute \ev\T.

\section{Signed Young Diagrams}
\label{SYD}
In this section, we study some combinatorial, geometric, and representation theoretic aspects of signed Young diagrams.
We also define and characterize the admissible signed Young diagrams.

%\subsection{The \texorpdfstring{$\nicefrac{\mathbb Z}2$}{Z/2} grading}
\subsection{The Ring \texorpdfstring{$R$}{R}}
\label{nonComRing}
Consider the non-commutative graded ring,
\begin{align*}
&& R=\nicefrac{\mathbb C\langle\alpha,\beta\rangle}{\left\langle\alpha^2,\beta^2\right\rangle}, 
&& \deg(\alpha)=\deg(\beta)=1.
\end{align*}
Let $M$ be a finite dimensional graded $R$-module,
with degrees the principal $\nicefrac{\mathbb Z}2$-space $\left\{\plus,\minus\right\}$,
and satisfying the following extra condition:
$\alpha$ (resp. $\beta$) acts trivially on $M(\minus)$ (resp. $M(\plus)$).
\begin{footnote}
{The modules $M$ discussed here are precisely the nilpotent representations of the $\widehat A_2$ quiver.
For a discussion of the general theory, see \cite{MR3526103}
}
\end{footnote}

Let $R^+$ (resp. $R^-$) denote the graded $R$-module whose underlying set is $R$,
with grading obtained by setting $\deg(1)=\plus$ (resp. $\minus$).
For $k\geq 1$, we have $k$-dimensional indecomposable graded modules,
\begin{align*}
U_k^+=\frac {R^+}{\left\langle \cdots xyx, y\right\rangle}, &&
U_k^-=\frac {R^-}{\left\langle \cdots yxy, x\right\rangle},
\end{align*}
%The monomials $\cdots xyx$ and $\cdots yxy$ are degree $k$.
%Let $U^+_1$ (resp. $U^-_1)$ be the graded one-dimensional $R$-module of degree \plus\ (resp. \minus),
We have a decomposition of graded modules,
\begin{align}
\label{colorDecomposition}
M=\bigoplus M_i,
\end{align}
where each $M_i$ is isomorphic to some $U_k^\pm$.

\begin{lemma}
\label{work2}
Let $M$ be an indecomposable $R$-module.
The element $z=\beta+\alpha\beta\in\gl(M)$ is nilpotent,
and its Jordan type is given by
\begin{align*}
J(z| U_{2n+1}^\pm) & =(n+1,1,\cdots,1),\\
J(z|U_{2n}^-)      & =(n+1,1,\cdots,1),\\
J(z| U_{2n}^+)     & =(n,1,\cdots,1).
\end{align*}
\end{lemma}
\begin{proof}
A simple calculation yields $z^k=\beta(\alpha\beta)^{k-1}+(\alpha\beta)^k$.
We decompose $M$ as a sum of an indecomposable $\mathbb C[z]$ sub-module and a subspace on which $z$ acts as $0$.
\begin{align*}
U_{2n}^+   & =\left\langle 1,\alpha,\alpha\beta ,\cdots,\alpha(\alpha\beta )^{n-1}\right\rangle\\
           & =\left\langle z^i\alpha\mid 0\leq i\leq n-1\right\rangle\oplus\left\langle (\alpha\beta )^i\mid 0\leq i\leq n-1\right\rangle.\\
U_{2n}^-   & =\left\langle 1,\beta,\beta\alpha ,\cdots,\beta(\beta\alpha )^{n-1}\right\rangle\\
           & =\left\langle z^i\mid 0\leq i\leq n-1\right\rangle\oplus\left\langle (\alpha\beta )^i\mid 1\leq i\leq n-1\right\rangle.\\
U_{2n+1}^+ & =\left\langle 1,\alpha,\alpha\beta ,\cdots,(\alpha\beta )^n\right\rangle\\
           & =\left\langle z^i\alpha\mid 0\leq i\leq n\right\rangle\oplus\left\langle (\alpha\beta )^i\mid 0\leq i\leq n-1\right\rangle.\\
U_{2n+1}^- & =\left\langle 1,\beta,\beta\alpha ,\cdots,(\beta\alpha )^n\right\rangle\\
           & =\left\langle z^i\mid 0\leq i\leq n\right\rangle\oplus\left\langle (\alpha\beta )^i\mid 1\leq i\leq n-1\right\rangle.
\end{align*}
Counting the dimensions of the summands, we obtain the desired result.
\end{proof}

\subsection{Signed Young Diagrams}
\label{sec:syd}
A \emph{$\pm$-filling} of a Young diagram is an assignment of \plus\ or \minus\ to each box of the Young diagram such that adjacent boxes in the same row have opposite signs.
The \emph{signature} of a $\pm$-filling is the integer pair $(q,p)$, where $q$ is the number of \plus\ boxes,
and $p$ is the number of \minus\ boxes.

We define an equivalence relation $\sim$ on $\pm$-fillings by setting two $\pm$-fillings to be equivalent if and only if one can be obtained from the other by some permutation of equi-sized rows.
The equivalence class of a $\pm$-filling under $\sim$ is called a \emph{signed Young diagram}.
\begin{figure}[h]
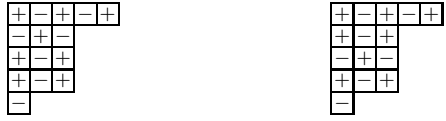

\centering
\ytableaushort{+-+-+,-+-,+-+,+-+,-}
\qquad\qquad\qquad\qquad
\ytableaushort{+-+-+,+-+,-+-,+-+,-}
\caption{Two equivalent $\pm$-fillings with signature $(8,7)$.}
\end{figure}

We denote by $SYD(q,p)$ the set of signed Young diagrams of signature $(q,p)$.
Given a signed Young diagram $\Lambda$ of signature $(q,p)$, we will denote by $\operatorname{sh}(\Lambda)$ the \emph{shape} of $\Lambda$, i.e., the Young diagram in $\Par(q+p)$ underlying $\Lambda$.

Consider an $R$-module $M$ as in \cref{nonComRing}.
We associate a signed Young diagram $\Lambda$ to $M$ as follows:
\begin{itemize}[leftmargin=0.8cm]
\item
If $M\cong U_k^+$ (resp. $M\cong U_k^-$),
we set $\Lambda$ to be a one-row signed Young diagram with $k$ boxes,
and set the rightmost box of the row to be \plus\ (resp. \minus).
\item
For general $M$, we construct $\Lambda$ as a union of the one-row diagrams corresponding to its indecomposable components,
see \cref{colorDecomposition}.
\end{itemize}
It is clear that signed Young diagrams classify the isomorphism classes of the graded $R$-modules described in \cref{nonComRing}.

Let $M(\plus)$ (resp. $M(\minus)$) be the $\plus$ (resp. $\minus$) graded component of $M$,
and set $q=\dim M(\plus)$ (resp. $p=\dim M(\minus)$).
The signature of $\Lambda$ is precisely $(q,p)$. %\dim M(\plus),\dim M(\minus))$,

\subsection{The Partitions \texorpdfstring{$\Lambda^\pm$}{LambdaPM}}
\label{LambdaPM}
The subspace $M(\plus)$ (resp. $M(\minus)$) is stable under the action of $\beta\alpha$ (resp. $\alpha\beta$).
We define the partitions,
\begin{align*}
\Lambda^+=J\left(\beta\alpha|M(\plus)\right)\in\Par(q), && \Lambda^-=J\left(\alpha\beta|M(\minus)\right)\in\Par(p). %&& \Lambda'\in\Par.
\end{align*}
Given the decomposition, $M=\oplus M_i$, see \cref{colorDecomposition},
we obtain a decomposition of $M(\plus)$ and $M(\minus)$ into indecomposables,
\begin{align*}
M(\plus)=\oplus (M_i\cap M(\plus)), && M(\minus)=\oplus (M_i\cap M(\minus)). 
\end{align*}
It follows that $\Lambda^+$ (resp. $\Lambda^-$) can be obtained from a $\pm$-filling in the equivalence class $\Lambda$
by removing all $\minus$ (resp. $\plus$) boxes,
(and possibly rearranging the rows if necessary).
This yields the Young diagram of $\Lambda^+$ (resp. $\Lambda^-$).

%Given a $\pm$-filling of $\lambda$, we may speak of the signature of $i^{th}$-row of this filling.
%However this is not a meaningful concept for a signed Young diagram.
%On the other hand, signature of the $i^{th}$-column of a signed Young diagram $\Lambda$ is well-defined.
%
\begin{lemma}
\label{nuLambda}
Let $\Lambda\in SYD(q,p)$ be the signed Young diagram associated to an $R$-module $M$ as in \cref{sec:syd}.
Consider $z\in\gl(M)$, given by $z=\beta+\alpha\beta$.
We can obtain the Jordan type of $z$ as follows.

Remove all \plus\ which do not appear in the first column of $\Lambda$,
then left-align (and possibly rearrange) the rows to obtain the partition $\lambda'$.
Let $\lambda\in\Par(q+p)$ be the partition obtained by attaching to $\lambda'$ a new row of size one for every \plus\ removed. 
\end{lemma}
\begin{proof}
A simple corollary to \cref{work2}.
\end{proof}

\subsection{The Variety \texorpdfstring{$\mathcal O$}{O}}
\label{defKH}
Fix positive integers $p$ and $q$.
Let $V_p$ (resp. $V_q$) be a $p$ (resp. $q$)-dimensional vector space, and let $H=\Hom(V_q,V_p)$.
We study the action of the group
$
K=GL(V_q)\times GL(V_p)
$
on $H$
%the vector space $H=\Hom(V_q,V_p)$, 
given by the formula
$
(g,h)\bullet x=hxg^{-1}
$.

%\subsection{The Cotangent Bundle of \texorpdfstring{$H$}{H}}
Let $H^\vee=\Hom(V_p,V_q)$.
The non-degenerate bilinear pairing,
\begin{align*}
\left\langle\ ,\ \right\rangle:H\times H^\vee\to\mathbb C, && \langle x,y\rangle=\operatorname{trace}(xy),
\end{align*}
yields the identification $T^*H=H\times H^\vee$; with this identification, the induced $K$-action on $T^*H$ is given by	
\begin{align*}%\label{Kaction2}
(g,h)\bullet (x,y)=(hxg^{-1},gyh^{-1}).
\end{align*}
We are interested in studying the $K$-action on the \emph{nilpotent cone},
%Let $\mathcal O$ be the $K$-stable subvariety of $T^*H$ given by
\begin{align}
\label{defO}
\mathcal O=\set{(x,y)\in H\times H^\vee}{xy\text{ is nilpotent}}.
\end{align}

%Next, consider $(x,y)\in\mathcal O$.
To each $(x,y)\in\mathcal O$, we associate a $\Lambda\in SYD(q,p)$.
Consider the graded $R$-module structure on $V=V_q\oplus V_p$, given by
$\alpha\mapsto x$, $\beta\mapsto y$, and
\begin{align*}
\deg(V_q)=\plus,&&\deg(V_p)=\minus.
\end{align*}
We set $\Lambda$ to be the signed Young diagram associated to this module, see \cref{sec:syd}.

%\subsection{Partial Order on Signed Young Diagrams}
Let $\mathcal O(\Lambda)$ be the set of all $(x,y)\in\mathcal O$ for which the associated signed Young diagram is $\Lambda$.
Following \cite{MR666395,MR2941520},
the $\mathcal(O(\Lambda))$ are precisely the $K$-orbits in $\mathcal O$,
i.e., $\nicefrac{\mathcal O}K=SYD(q,p)$.
% $\mathcal O(\Lambda)\subset\mathcal O$
Further, we have $(x,y)\in\mathcal O(\Lambda)$ if and only if% given by the conditions:
\begin{gather*}
J\left(yx|V_q\right)  =\Lambda^+,\qquad\qquad J\left(xy|V_p\right) =\Lambda^-,\\
\dim\ker(x(yx)^i|V_q)  =\#\left\{\plus\in\Lambda\text{ contained in the first }2i+1\text{ columns}\right\},\\
\dim\ker(y(xy)^i|V_p)  =\#\left\{\minus\in\Lambda\text{ contained in the first }2i+1\text{ columns}\right\}.
\end{gather*}
We will denote the \emph{closure inclusion order} on $\nicefrac{\mathcal O}K$ by $\preceq$, i.e.,
\begin{equation}
\label{sydPartial}
\Lambda\preceq\Upsilon\iff\overline{\mathcal O(\Lambda)}\subset\overline{\mathcal O(\Upsilon)}.
\end{equation}
%For a combinatorial description of $\preceq$, see \cite{johnson}.
Following \cite{MR666395}, the closure $\overline{\mathcal O(\Lambda)}\subset\mathcal O$ is given by the following conditions:
\begin{gather*}
J\left(yx|V_q\right)\preceq\Lambda^+,\qquad\qquad J\left(xy|V_p\right)\preceq\Lambda^-,\\
\dim\ker(x(yx)^i|V_q)\geq\#\left\{\plus\in\Lambda\text{ contained in the first }2i+1\text{ columns}\right\},\\
\dim\ker(y(xy)^i|V_p)\geq\#\left\{\minus\in\Lambda\text{ contained in the first }2i+1\text{ columns}\right\}.
\end{gather*}
In particular, the map $\Lambda\mapsto(\Lambda^+,\Lambda^-)$ is a poset homomorphism.

\begin{lemma}
[\cite{MR1168491}] 
\label{dimMlambda}
%$\mathcal O(\Lambda)$ is a \emph{Lagrangian} subvariety of $\mathcal N(\operatorname{sh}(\Lambda)$, hence we have,
%\begin{align}
We have $\dim\mathcal O(\Lambda)=\nicefrac12\dim\Ni(\sh\Lambda)$.
%\end{align}
\end{lemma}

\begin{lemma}
\label{specialDegen}
%The map $\operatorname{sh}$ is a poset homomorphism, i.e, we have,
The map $\operatorname{sh}:SYD(q,p)\to\Par(q+p)$ is a poset homomorphism, i.e, 
$
\Lambda\preceq\Upsilon\implies\operatorname{sh}(\Lambda)\preceq\operatorname{sh}(\Upsilon).
$
\end{lemma}
\begin{proof}
Consider $\Lambda\preceq\Upsilon$, so that
\begin{align*}
\mathcal O(\Lambda)\subset\overline{\mathcal O(\Upsilon)}\subset\overline{\mathcal N(\operatorname{sh}(\Upsilon))}.
\end{align*}
Observe that $\overline{\mathcal N(\operatorname{sh}(\Upsilon))}$ is $G$-stable; hence we have,
\begin{align*}
\mathcal N(\operatorname{sh}(\Lambda))=G\cdot\mathcal O(\Lambda)\subset\overline{\mathcal N(\operatorname{sh}(\Upsilon))},
\end{align*}
which is equivalent to $\operatorname{sh}(\Lambda)\preceq\operatorname{sh}(\Upsilon)$.
\end{proof}

\subsection{The Moment Map \texorpdfstring{$\mu_H$}{muH}}
\label{momentH}
The $K$-action on $H$ admits a moment map,
\begin{align*}
\mu_H: H\times H^\vee  & \rightarrow\gl(V_q)\times\gl(V_p),\\
\mu_H(x,y)                             & =(yx,xy).
\end{align*}
%is given by $\mu_H(x,y)=(yx,xy)$.
%Let $\mathcal M$ denote  of $K$. %, i.e., the variety of nilpotent matrices in $\gl(V_q)\times\gl(V_p)$.
The $K$-orbit decomposition of the nilpotent cone $\mathcal M$ of $K$ is given by
\begin{align*}
\mathcal M=\bigsqcup\limits_{\nu\vdash q,\,\lambda\vdash p}\mathcal M(\nu,\lambda), && \mathcal M(\nu,\lambda)=N(\nu)\times\mathcal N(\lambda).
\end{align*}
It is clear from \cref{LambdaPM} that for any $\Lambda\in SYD(q,p)$, we have, 
%Since $\mu_H$ is $K$-equivariant, 
\begin{align*}
\mu_H(\mathcal O(\Lambda))=\mathcal M(\Lambda^+,\Lambda^-).
\end{align*}
In particular, we have $\mu_H(\mathcal O)\subset\mathcal M$. 

%The map $\Lambda\mapsto(\Lambda^+,\Lambda^-)$ is a poset homomorphism.
%To obtain $\Lambda^+$ (resp. $\Lambda^-$) from $\Lambda$, we remove all $\minus$ (resp. $\plus$) boxes, left-align each row, and then rearrange the rows in weakly decreasing order if necessary, to obtain the Young diagram of $\Lambda^+$ (resp. $\Lambda^-$).
%This can be expressed by following formulae:
%\begin{align*}
%\Lambda^+&=\lceil\nicefrac{\Lambda_+}2\rceil\sqcup\lfloor\nicefrac{\Lambda_-}2\rfloor\\
%\Lambda^-&=\lceil\nicefrac{\Lambda_-}2\rceil\sqcup\lfloor\nicefrac{\Lambda_+}2\rfloor
%\end{align*}
%Here $\lfloor\ \rfloor$ and $\lceil\ \rceil$ are the floor and ceiling function respectively.

\begin{lemma}
\label{sydCompare}
For any $\Lambda\in SYD(q,p)$, we have,
\begin{equation*}
%\operatorname{sh}(\Lambda)=\Lambda_+\sqcup\Lambda_-\preceq\Lambda^++\Lambda^-.
\operatorname{sh}(\Lambda)\preceq\Lambda^++\Lambda^-.
\end{equation*}
Equality holds if and only if %no odd integer $i$ occurs as a part of both $\Lambda_+$ and $\Lambda_-$, or equivalently, 
for each odd integer $i$, all the rows of size $i$ in $\Lambda$ have the same signature.
%all rows in $\Lambda$ of any given odd size have the same signature.
\end{lemma}
\begin{proof}
Fix some $\pm$-filling corresponding to $\Lambda$, and let $a^+_i$ (resp. $a^-_i$) be the number of $\plus$ boxes (resp. $\minus$) boxes in the $i^{th}$ row of this filling.
We have,
\begin{equation*}
\sh\Lambda=(a^+_1+a^-_1,a^+_2+a^-_2,\cdots)
\end{equation*}
Now, observe that the sequence $(a^\pm_1,a^\pm_2,\cdots)$ is some permutation of the weakly decreasing sequence $(\Lambda^\pm_1,\Lambda^\pm_2,\cdots)$, and hence, we have,
\begin{align}
\label{sydWork0}
&&\Lambda^\pm(1)+\cdots+\Lambda^\pm(i)\geq a^\pm(1)+\cdots+a^\pm(i),&&\forall i.
\end{align}
This yields the claimed inequality,
$\operatorname{sh}(\Lambda)\preceq\Lambda^++\Lambda^-$,
%\begin{equation*}
%\operatorname{sh}(\Lambda)=\Lambda_+\sqcup\Lambda_-\preceq\Lambda^++\Lambda^-,
%\end{equation*}
with equality holding if and only if equality holds for all $i$ in \cref{sydWork0}. 
This happens if and only if the sequences $(a^+_1,\ldots)$ and $(a^-_1,\ldots)$ are weakly decreasing,
which happens if and only if any two consecutive rows of the same odd length in the $\pm$-filling have the same signature.
\end{proof}

%\begin{definition}
\subsection{Admissible Signed Young Diagrams}
\label{defn:admissible}
We call ${\Lambda}\in SYD(q,p)$ \emph{admissible} if $\operatorname{sh}(\Lambda)=\Lambda^++\Lambda^-$, and we denote by $ASYD(q,p)$ the set of all admissible signed Young diagrams of signature $(q,p)$.
%\end{definition}

\begin{theorem}
\label{main:admissible}
%Let $\Lambda$ be a $2$-filling of signature $(q,p)$. %and let $(\nu,\lambda)=\Theta(\Lambda)$.
The following conditions are equivalent for $\Lambda\in SYD(q,p)$:
\begin{enumerate}
\item $\Lambda$ is admissible.
\item $2\dim\mathcal O(\Lambda)=2|\Lambda^+||\Lambda^-|+\dim\mathcal M(\Lambda^+,\Lambda^-)$.
\item $\Lambda$ is maximal in $\left\{\Upsilon\in SYD(q,p)\mid(\Upsilon^+,\Upsilon^-)=(\Lambda^+,\Lambda^-)\right\}$.
\item $\overline{\mathcal O(\Lambda)}$ is an irreducible component of $\overline{\mu_H^{-1}(\mathcal M(\Lambda^+,\Lambda^-))}$.
\end{enumerate}
\end{theorem}
\begin{proof}
$(1)\implies (2)$:
Suppose $(1)$ holds.
%, i.e., $\sh\Lambda=\Lambda^++\Lambda^-$.
We see from \cref{dimMlambda} that 
\begin{equation*}
2\dim\mathcal O(\Lambda)=\dim\mathcal N(\sh\Lambda)=\dim\mathcal N(\Lambda^++\Lambda^-).
\end{equation*}
%Now, $(2)$ follows from \cref{magic}.
Using \cref{macMagic}, we compute,
\begin{equation*}
\begin{split}
2\dim\mathcal O(\Lambda)  & =(q+p)(q+p+1)-2\sum i(\Lambda^+(i)+\Lambda^-(i))\\
                          & =p(p+1)+p(p+1)+2qp-2\sum i\Lambda^+(i)-2\sum i\Lambda^-(i)\\
                          & =2pq+\dim\Ni(\Lambda^+)+\dim\Ni(\Lambda^-)\\
                          & =2|\Lambda^+||\Lambda^-|+\dim\mathcal M(\Lambda^+,\Lambda^-).
\end{split}
\end{equation*}

$(2)\implies(3)$:
Suppose $(2)$ holds, and $(3)$ does not hold, i.e., there exists $\Upsilon\succ\Lambda$ satisfying
$(\Upsilon^+,\Upsilon^-)=(\Lambda^+,\Lambda^-)$.
%We see from \cref{sydCompare} that $\operatorname{sh}(\Lambda)\preceq\Lambda^++\Lambda^-$.
Using \cref{dimMlambda,sydCompare}, we deduce a contradiction:
%\implies 2\dim\mathcal O(\Lambda) & <
\begin{align*}
2\dim\mathcal O(\Lambda) & <2\dim\mathcal O(\Upsilon)=\dim\mathcal N(\operatorname{sh}(\Upsilon))
                          \leq 2 \dim\Ni(\Upsilon^++\Upsilon^-)\\
                         & =2pq+\dim\mathcal M(\Upsilon^+,\Upsilon^-)
                          =2pq+\dim\mathcal M(\Lambda^+,\Lambda^-).
\end{align*}

$(3)\implies(1)$:
Suppose $(1)$ does not hold.
It follows from \cref{sydCompare} that some $\pm$-filling corresponding to $\Lambda$ contains two consecutive rows of the same odd length with opposite signs.
%\begin{align*}
%\left\{\Upsilon\in SYD(q,p)\mid(\Upsilon^+,\Upsilon^-)=(\Lambda^+,\Lambda^-)\right\},
%\end{align*}
We construct $\Upsilon$, by moving a box in $\Lambda$, (and rearranging some rows if necessary), see below:
\begin{center}
\begin{tikzcd}
\ytableaushort{+-+-+,-+-+-}\arrow[r,mapsto]&\ytableaushort{+-+-+-,-+-+}
\end{tikzcd}
\end{center}
It is clear that $(\Upsilon^+,\Upsilon^-)=(\Lambda^+,\Lambda^-)$. 
Comparing with \cref{sydPartial}, we deduce $\Lambda\prec\Gamma$.
We deduce that if $(1)$ does not hold, then neither does $(3)$.

$(3)\iff(4)$:
This is a direct consequence of the fact that $\preceq$ is the closure inclusion order.
\end{proof}

\begin{corollary}
\label{asydBoxes}
Consider $\Lambda\in ASYD(q,p)$.
For any $\pm$-filling in the equivalence class $\Lambda$, the following is true:
The number of $\plus$ (resp. \minus) in the $i^{th}$-row is $\Lambda^+(i)$ (resp. $\Lambda^-(i)$).
\end{corollary}

\section{Matrix Schubert Varieties}
\label{MSV}
Let $V_q$, $V_p$, $H=\Hom(V_q,V_p)$, $K=GL(V_q)\times GL(V_p)$, and $\mathcal O$ be as in \cref{defKH}.
Set $X=X_q\times X_p$, where
\begin{align*}
X_q=\set{F_\bullet=(F_1\subset \cdots\subset F_q\subset V_q)}{\dim F_i=i},\\
X_p=\set{F_\bullet=(F_1\subset \cdots\subset F_p\subset V_p)}{\dim F_i=i},
\end{align*}
are the variety of full flags in $V_q$, $V_p$ respectively.

\subsection{Partial Permutations}
A \emph{partial permutation} $\tau$ is a matrix with entries in the set $\{0,1\}$ in which every row and every column contains at most one non-zero entry.
We denote by $\mathcal{PP}(p,q)$ the set of partial permutations of size $p\times q$.

\subsection{Matrix Schubert Varieties}
\label{fulton}
Consider the variety,
\begin{align*}
C=X\times H= X_q\times X_p\times \Hom(V_q,V_p),
\end{align*}
along with the $K$-action given by,
\begin{align*}
(g,h)\bullet(E_\bullet,F_\bullet,x) =(gE_\bullet,hF_\bullet,hxg^{-1}).
\end{align*}
%The $B_K$-orbit clousures are called \emph{matrix Schubert varieties}.
For $\tau$ a partial permutation of size $p\times q$, let
\begin{equation*}
C(\tau)=\set{(E_\bullet,F_\bullet,x)\in C}{\dim (xE_i+F_j)=j+\sum\limits_{l\leq i}\sum\limits_{k>j}\tau_{kl}},
\end{equation*}
%where $\tau(j,i)$ denotes the south-west submatrix of $\tau$ of size $j\times i$.
Following \cite{MR1154177}, the $C(\tau)$ are precisely the $K$-orbits in $C$, so that we have,
\begin{equation*}
\nicefrac CK=\mathcal{PP}(p,q).
\end{equation*}
The closure $\overline{C(\tau)}$ is called a \emph{matrix Schubert variety};
we have,% it has the following description:
\begin{equation*}
\overline{C(\tau)}=\set{(E_\bullet,F_\bullet,x)\in C}{\dim(xE_i+F_j)\leq j+\sum\limits_{l\leq i}\sum\limits_{k>j}\tau_{kl}}.
\end{equation*}

\subsection{The Steinberg Variety for Matrix Schubert Varieties}
\label{steinMSV}
Following \cref{sub:flags,springerMap,momentH}, the $K$-action on $C$ admits a moment map,
\begin{align*}
\mu_C:T^* X_q\times T^* X_p\times T^*H            & \to\gl(V_q)\times\gl(V_p),\\
\mu_C((E_\bullet,\alpha),(F_\bullet,\beta),(x,y)) & =(\alpha+yx,\beta+xy).
\end{align*}
We deduce that the \emph{Steinberg variety}, $Z_C=\mu_C^{-1}(0)$, has the following description:
\begin{align*}
Z_C & =\set{(E_\bullet,F_\bullet,x,y)\in X\times\mathcal O}{yx E_i\subset E_{i-1},\ xy F_i\subset F_{i-1},\ \forall i}\\
    & =\set{(E_\bullet,F_\bullet,x,y)\in X\times\mathcal O}{(E_\bullet,yx)\in T^* X_q,\,(F_\bullet,xy)\in T^* X_p}.
\end{align*}
For $\tau\in\mathcal{PP}(p,q)$, let $Z_C(\tau)=T^*_CC(\tau)$, i.e., the conormal bundle of $C(\tau)$.
Following \cref{kashiwara}, we have,
\begin{equation*}
\operatorname{Irr}\left(Z_C\right)=\set{\overline{Z_C(\tau)}}{\tau\in\nicefrac CK}.
\end{equation*}

\subsection{The Map \texorpdfstring{$\pr$}{pr}}
The $K$-equivariant map,
\begin{align*}
\pr:Z_C\to\mathcal O,&& (E_\bullet,F_\bullet,x,y)\mapsto (x,y),
\end{align*}
is proper; 
in particular, it sends the $K$-stable irreducible component $\overline{Z_C(\tau)}$ to a closed irreducible $K$-stable subvariety of $\mathcal O$.
Since $\nicefrac{\mathcal O}K$ is finite, this subvariety is necessarily a $K$-orbit closure.
We obtain a map, $\pr:\nicefrac CK\to\nicefrac{\mathcal O}K$,
characterized by 
\begin{equation*}
\pr(\overline{Z_C(\tau)})=\overline{\mathcal O(\pr(\tau))}.
\end{equation*}
Recall the identifications $\nicefrac CK=\mathcal {PP}(p,q)$ and $\nicefrac{\mathcal O}K=SYD(q,p)$.
With an abuse of notation, we also denote the induced map,
$
\pr:\mathcal{PP}(p,q)\to SYD(q,p).
$

\subsection{The Fibre Bundle \texorpdfstring{$Z_C(\Lambda)\to\mathcal O(\Lambda)$}{Z(Lambda)}}
\label{ZLambdaFibreBundle}
Recall from \cref{steinDouble,momentH}, the moment maps,
\begin{align*}
&&\mu_X:T^*X\to\mathcal M && \text{and} && \mu_H:T^*H\to\mathcal M.&&
\end{align*} 
Given $\Lambda\in SYD(q,p)$, let $Z_C(\Lambda)=\pr^{-1}(\mathcal O(\Lambda))$, and consider the induced map,
\begin{equation*}
\pr:Z_C(\Lambda)\to\mathcal O(\Lambda).
\end{equation*}
We see from \cref{steinMSV} that the fibre of $\pr$ over the point $(x,y)\in\mathcal O(\Lambda)$ is precisely
\begin{equation*}
\mu_{X_q}^{-1}(yx)\times\mu_{X_p}^{-1}(xy)=\mu_X^{-1}(yx,xy)=\mu_X^{-1}(\mu_H(x,y)).
\end{equation*}
It follows that
\begin{align*}
\operatorname{Irr}(Z_C(\Lambda))\cong SYT(J(yx))\times SYT(J(xy))=SYT(\Lambda^+)\times SYT(\Lambda^-).
\end{align*}
Further, the pure-dimensionality of Springer fibres, see \cref{springerFibre},
implies that $Z_C(\Lambda)$ is a pure-dimensional fibre bundle over $\mathcal O(\Lambda)$.
%Following \cref{dimSpringerFibre}, we compute,
%\begin{equation}
%\label{eqnFibreBundle}
%\begin{split}
%\dim Z_C(\Lambda)-\dim\mathcal O(\Lambda) & =\dim\mu_{X_q}^{-1}(yx)+\dim\mu_{X_p}^{-1}(xy)\\
%				                          & =\left(\dim X_q-\nicefrac12\dim\Ni(\Lambda^+)\right)+\left(\dim X_p-\nicefrac12\dim\Ni(\Lambda^-)\right)\\
%				                          & =\dim X-\nicefrac12\dim\Ni(\Lambda^+)-\nicefrac12\dim\Ni(\Lambda^-).
%\end{split}
%\end{equation}

\begin{theorem}
\label{LambdaAdmissible}
The image of the map $\pr:\mathcal{PP}(p,q)\to SYD(q,p)$ is precisely the set of admissible signed Young diagrams.
\end{theorem}
\begin{proof}
%Suppose $\Lambda\in\pr(\mathcal{PP}(p,q))$, i.e., $\overline{\mathcal O(\pr(\tau))}=\pr\left(\overline{T^*_CC_\tau}\right)$ for some $\tau\in\mathcal{PP}(p,q)$.
%For $\Lambda\in SYD(q,p)$, set $Z_C(\Lambda)=\pr^{-1}(\mathcal O(\Lambda))$.
Observe that $\Lambda\in\pr(\mathcal{PP}(p,q))$ if and only if $\overline{Z_C(\Lambda)}$ contains at least one irreducible component of $Z_C$.
Since $Z_C$ is pure-dimensional, this is equivalent to 
\begin{equation*}
\begin{split}
\dim\overline{Z_C(\Lambda)}=\dim Z_C=\dim C=\dim X+\dim H.
\end{split}
\end{equation*}
Following \cref{ZLambdaFibreBundle,dimSpringerFibre,LambdaPM}, we compute,
\begin{equation*}
\begin{split}
\dim\overline{Z_C(\Lambda)} & =\dim Z_C(\Lambda)=\dim\mathcal O(\Lambda)+\dim\mu_{X_q}^{-1}(yx)+\dim\mu_{X_p}^{-1}(xy)\\
				             & =\dim\mathcal O(\Lambda)+\left(\dim X_q-\nicefrac12\dim\Ni(\Lambda^+)\right)+\left(\dim X_p-\nicefrac12\dim\Ni(\Lambda^-)\right)\\
				             & =\dim\mathcal O(\Lambda)+\dim X-\nicefrac12\dim\Ni(\Lambda^+)-\nicefrac12\dim\Ni(\Lambda^-).
\end{split}
\end{equation*}
It follows that $\Lambda\in\pr\left(\mathcal{PP}(p,q)\right)$ if and only if 
\begin{align*}
     & \dim\mathcal O(\Lambda)-\nicefrac12\dim\Ni(\Lambda^+)-\nicefrac12\dim\Ni(\Lambda^-)=\dim H,\\
\iff & 2\dim\mathcal O(\Lambda)={\dim\Ni(\Lambda^+)}+\dim\Ni(\Lambda^-)+2pq,
\end{align*}
which is one of the characterizations of admissibility, see \Cref{main:admissible}.
\end{proof}

We now have all the pieces needed to prove the main result of this section.
\begin{mainTheorem}
\label{main:bijection}
Consider $\tau\in\mathcal{PP}(p,q)$, and let
$(E_\bullet,F_\bullet,x,y)$ be a generic point in $C(\tau)$.
Let $\Lambda=\pr(\tau)$, $\Q=\Tab(yx,E_\bullet)$, and $\P=\Tab(xy,F_\bullet)$.
We have a bijection,
\begin{equation*}
\mathcal{PP}(p,q)=\bigsqcup\limits_{\Lambda\in ASYD(q,p)}SYT(\Lambda^+)\times SYT(\Lambda^-),
%{(\Lambda,\Q,\P)}{\Lambda\in ASYD(q,p),\,\Q\in SYT(\Lambda^+),\,\P\in SYT(\Lambda^-)}.
\end{equation*}
given by $\tau\mapsto (\Lambda,\Q,\P)$.
\end{mainTheorem}
\begin{proof}
%Following \cref{kashiwara,fulton}, we see that $\operatorname{Irr}(Z_C)$ is indexed by $\mathcal{PP}(p,q)$.
The bijection $\operatorname{Irr}(Z_C)=\mathcal{PP}(p,q)$ is a consequence of \cref{kashiwara,fulton}.
%\subsection{Proof of \texorpdfstring{\cref{main:bijection}}{main:bijection}}
%Let $Z_C(\tau)$ be the conormal variety of the $K$-orbit $C(\tau)$.
Now consider $\tau\in\mathcal{PP}(p,q)$, and let $\Lambda=\pr(\tau)$.
It follows from \cref{LambdaAdmissible} that $\Lambda\in ASYD(q,p)$.
Combined with \cref{ZLambdaFibreBundle}, this yields
\begin{equation*}
\operatorname{Irr}(Z_C)=\bigsqcup\limits_{\Lambda\in ASYD(q,p)}SYT(\Lambda^+)\times SYD(\Lambda^-).
\end{equation*}
Finally, we see from \cref{springerFibre} that the indexing of $\operatorname{Irr}(Z_C)$
is given precisely by the conditions $\Tab(yx,E_\bullet)=\Q$ and $\Tab(xy,F_\bullet)=\P$.
\end{proof}

\section{Combinatorial Description of the Bijection}
Let $V_q$ and $V_p$ be as in \cref{SYD}, and set $V=V_q\oplus V_p$.
We identify $V_q$ and $V_p$ with the subspaces $V_q\oplus0$ and $0\oplus V_p$ of $V$ respectively.
We fix $\underline a=(1,\cdots,1,p)$ and $\underline b=(q,1,\cdots,1)$, compositions of $N=p+q$, and denote
\begin{align*}
Y&=\fl aV\times\fl bV.
\end{align*}
%Let $G=GL(V)$.
We identify $K=GL(V_q)\times GL(V_p)$ as a subgroup of $G=GL(V)$ via the natural diagonal embedding.
As in \cref{MSV}, we set
\begin{align*}
X_q  & =\Fl(V_q),      & X_p    & =\Fl(V_p),      & X & =X_q\times X_p,\\
   H & =\Hom(V_q,V_p), & H^\vee & =\Hom(V_p,V_q), & C & =X\times H.
\end{align*}

In this section, we prove \cref{main:PPtoTriple,main:tripleToPP}.
Our main tool is \cref{relateRSK}, which relates the Steinberg variety $Z_Y$ of $Y$ with the Steinberg variety $Z_C$ of $C$.

\subsection{The Embedding \texorpdfstring{$C\into Y$}{C into Y}}
Given $x\in\Hom(V_q,V_p)$, let $\widetilde x\in\gl(V)$ be the invertible linear map given by
\begin{equation*}
\widetilde x(v,w)=(v,xv+w).
\end{equation*}
In matrix form, we have,
\begin{align*}
\widetilde x=
\begin{pmatrix}
I_q & 0\\
x & I_p
\end{pmatrix}
,&&\widetilde x^{-1}=
\begin{pmatrix}
I_q & 0\\
-x  & I_p
\end{pmatrix},
\end{align*}
where $I_q$ and $I_p$ are identity matrices.
Next, for $(E_\bullet,F_\bullet)\in X_q\times X_p$, let 
\begin{align*}
\widetilde E_\bullet=(\widetilde E_1,\cdots,\widetilde E_{q+1}),&&\widetilde F_\bullet=(\widetilde F_0,\cdots,\widetilde F_p),
\end{align*}
be the partial flags given by
\begin{align*} 
\widetilde E_i=
\begin{cases}
E_i &\text{for }1\leq i\leq q,\\
V   &\text{for }i=q+1,
\end{cases}&&
\widetilde F_i=
\begin{cases} V_q&\text{for }i=0,\\
V_q\oplus F_i   &\text{for }1\leq i\leq p.
\end{cases}
\end{align*}
Observe that the embedding, 
\begin{align}
\label{mapEmbed}
\phi:C\into Y,&& (E_\bullet,F_\bullet,x)\mapsto (\widetilde x\widetilde E_\bullet,\widetilde F_\bullet),
\end{align}
is $K$-equivariant.
This is most easily seen via matrix calculation,
\begin{align*}
\begin{pmatrix}
g&0\\
0&h\\
\end{pmatrix}
\begin{pmatrix}
I_q & 0\\
x & I_p
\end{pmatrix}
\begin{pmatrix}
g&0\\
0&h\\
\end{pmatrix}
^{-1}
=
\begin{pmatrix}
I_q & 0\\
hxg^{-1}  & I_p\\
\end{pmatrix}.
\end{align*}

\subsection{Induced Map on Orbits} 
\label{hatMap}
Let $\nicefrac CK\into\nicefrac YK$ be the map on $K$-orbits induced from $\phi$,
and let $\nicefrac YK\to\nicefrac YG$ be the natural projection. 
We have a composite map,
\begin{align*}
\widehat{\ }\,:\nicefrac CK\into\nicefrac YK\to\nicefrac YG, && \tau\mapsto\widehat\tau.
\end{align*}
Here $\tau\in \mathcal{PP}(p,q)$,
and $\widehat\tau$ is a $(p+1)\times(q+1)$ matrix satisfying the conditions described in \cref{sub:schubert}.

We describe this map in some detail.
Fix a basis $e_1,\cdots,e_q$ of $V_q$, and a basis $e_{q+1},\cdots,e_{q+p}$ of $V_p$.
Given $\tau\in \mathcal{PP}(p,q)$, we interpret $\tau$ as an element of $\Hom(V_q,V_p)$ with respect to the basis $e_1,\cdots,e_q$ and $e_{q+1},\cdots,e_{q+p}$.
Then, we have $(E_\bullet,F_\bullet,\tau)\in C_\tau$,
where $E_\bullet$ and $F_\bullet$ are the flags given by
\begin{align*}
%\label{flagBasis}
E_i=\left\langle e_1,\cdots,e_i\right\rangle, && F_j=\left\langle e_{q+1},\cdots,e_{q+j}\right\rangle.
\end{align*}
We can now directly compute the matrix $\widehat\tau$.
% corresponding to the $G$-orbit of $\phi(E_\bullet,F_\bullet,\tau)$, see \cref{sub:schubert}.
It is obtained from $\tau$ by adjoining a row to the top and a column to the right,
\begin{equation}
\label{tauHat}
\widehat\tau=
\begin{pmatrix}
 l_1       & \cdots & l_q       & r\\
 \tau_{11} & \cdots & \tau_{1q} & m_1\\
 \vdots    & \ddots & \vdots    & \vdots\\
 \tau_{p1} & \cdots & \tau_{pq} & m_p
 \end{pmatrix}
\end{equation}

The entries in the top row and rightmost column of $\widehat\tau$ are uniquely determined by the conditions described in \cref{sub:schubert}. 
In particular, we have
\begin{align*}
l_j=1-\sum\limits_i\tau_{ij}, && m_i=1-\sum\limits_j\tau_{ij}, && r=\rk(\tau)=\sum\tau_{ij}.
\end{align*}
Observe that $l_j$ (resp. $m_i$) is $0$ if the $j^{th}$-column (resp. $i^{th}$-row) of $\tau$ contains a non-zero entry, and $1$ otherwise.

%\begin{corollary}
%The map $\tau\mapsto\widehat\tau$ is a bijection.
%In particular, 
%following \cref{steinDouble}, 
As a corollary, we deduce that 
\begin{equation}
\label{hatBij}
\operatorname{Irr}(Z_Y)=\set{Z_Y(\widehat\tau)}{\tau\in\mathcal{PP}(p,q)}.
\end{equation}
%\end{corollary}
%\begin{proof}
Indeed, the above calculation shows that the map $\tau\mapsto\widehat\tau$ is a bijection between $\mathcal{PP}(p,q)$ and $\nicefrac YG$.
The latter indexes $\operatorname{Irr}(Z_Y)$, see \cref{kashiwara}.
%\end{proof}

\begin{proposition}
\label{relateRSK}
Consider $(E_\bullet,F_\bullet,x)\in C$,
and $z\in\gl(V)$ given by
\begin{align*}
z=
\begin{pmatrix}\alpha   &\beta\\ \gamma &\delta \end{pmatrix}
\in
\begin{pmatrix}\gl(V_q)&\Hom(V_p,V_q)\\\Hom(V_q,V_p)&\gl(V_p)\end{pmatrix}
=\gl(V).
\end{align*}
Then $(\widetilde x\widetilde E_\bullet,\widetilde F_\bullet,z)\in Z_Y$, if and only if
\begin{align*}
&& && \alpha=\gamma=\delta-x\beta=0, && \text{and } && (E_\bullet,F_\bullet,x,\beta)\in Z_C. && &&
\end{align*}
In particular, the fibre of $Z_Y(\widehat\tau)\to Y$ at the point $(\widetilde x\widetilde E_\bullet,\widetilde F_\bullet)$ is isomorphic to the fibre of $Z_C(\tau)\to C$ at the point $(E_\bullet,F_\bullet,x)$.
\end{proposition}
\begin{proof}
Following \cref{steinDouble}, we have $(\widetilde x\widetilde E_\bullet,\widetilde F_\bullet,z)\in Z_Y$ if and only if
\begin{align}
\label{condOnG} z\widetilde F_i\subset\                                           & \widetilde F_{i-1},                                                    &  & \text{for }0\leq i\leq p,\\
\label{condOnF} z\widetilde x\widetilde E_i\subset\widetilde x\widetilde E_i \iff & \widetilde x^{-1}z\widetilde x\widetilde E_i\subset\widetilde E_{i-1}, &  & \text{for }1\leq i\leq q+1.
\end{align}
\Cref{condOnG} is equivalent to 
\begin{align}
\label{cond1Ct}
    &&\alpha=\gamma=0,  &&\text{and}&&\delta F_i\subset F_{i-1},\ \forall\,1\leq i\leq p.   &&
\end{align}
For \cref{condOnF}, we compute,
\begin{align*}
\widetilde x^{-1}z\widetilde x=
\begin{pmatrix}
1&0\\
-x&1
\end{pmatrix}
\begin{pmatrix}
0&\beta\\
0&\delta
\end{pmatrix}
\begin{pmatrix}
1&0\\
x&1
\end{pmatrix}
=
\begin{pmatrix}
\beta x&\beta\\
-x\beta x+\delta x&-x\beta+\delta
\end{pmatrix}.
\end{align*}
It follows that \cref{condOnF} is equivalent to 
%This yields $\delta=x\beta$, and
\begin{align}
\label{cond2Ct}
    &&\delta=x\beta,&&\text{and}&&\beta xE_i\subset E_i,\ \forall\,1\leq i\leq q.&&
\end{align} 
We see that \cref{condOnG,condOnF} hold if and only if \cref{cond1Ct,cond2Ct} hold.
\end{proof}

\begin{mainTheorem}
%\label{main:PPtoTriple}
Consider $\tau\in\mathcal{PP}(p,q)$, set $\Lambda=\pr(\tau)$, 
and let
\begin{align*}
(\Q,\P)\in SYT(\Lambda^+)\times SYT(\Lambda^-)
\end{align*}
be the tableaux pair corresponding to $\tau$ via the bijection in \cref{main:bijection}.
%The bijection of \cref{main:bijection} is given by
Let $\widehat\tau$ be given by \cref{tauHat}, and let $(\widehat\Q,\widehat\P)$ be the tableaux pair corresponding to $\widehat\tau$ via the Robinson-Schensted-Knuth correspondence.
Let $\lambda=\sh{\widehat\Q}=\sh{\widehat\P}$.
Then
\begin{align*}
\Q=\widehat\Q(q),&& \P=\operatorname{Rect}(\widehat\P,q), &&\sh\Lambda=\sh\Q+\sh\P,
\end{align*}
and a $\pm$-filling in the equivalence class $\Lambda$ is given by the following rule:
%the first box of the $i^{th}$ row is $\plus$ if either $\sh\Q(i)>\sh\P(i)$ or $\sh\Q(i)=\sh\P(i)>\lambda(i)$;
the first box of the $i^{th}$ row is $\plus$ if $\sh\P(i)<\lambda(i)$;
otherwise it is $\minus$.
\end{mainTheorem}
\begin{proof}
Consider $(E_\bullet,F_\bullet,x, y)\in Z_C(\tau)$ for some $\tau\in \mathcal{PP}(p,q)$, and set
%Let $\lambda$ be the Jordan type of
\begin{equation*}
z=
\begin{pmatrix}
0 & y\\ 
0 &x y
\end{pmatrix}
\in
\begin{pmatrix}\gl(V_q)&\Hom(V_p,V_q)\\\Hom(V_q,V_p)&\gl(V_p)\end{pmatrix}
=\gl(V).
\end{equation*}

Let $\widehat\tau$ be given by \cref{tauHat},
and let $(\widehat\Q,\widehat\P)$ be the tableaux pair corresponding to the matrix $\widehat\tau$ via the Robinson-Schensted-Knuth correspondence.
Recall that $\widehat{\Q}\in SYT(\lambda,\underline a)$, and $\widehat{\P}\in SYT(\lambda,\underline b)$, where $\lambda$ is the Jordan type of $z$.

Let $Z'$ be the fibre of $Z_Y(\widehat\tau)$ over the point $(\widetilde x\widetilde E_\bullet,\widetilde F_\bullet)$.
Recall the dense open subset $Z_Y^\circ(\widehat\tau)\subset Z_Y(\widehat\tau)$ from \cref{ZYcirc}.
Since $Y(\widehat\tau)$ is $G$-homogeneous, the intersection $Z_Y^\circ(\widehat\tau)\cap Z'$ is a dense open subset of $Z'$.
Following \cref{relateRSK}, $Z'$ is isomorphic to the fibre of $Z_C(\tau)\to C$ at the point $(E_\bullet,F_\bullet,x)$; %see \cref{relateRSK};
hence for \emph{generic} $(E_\bullet,F_\bullet,x, y)\in Z_C(\tau)$,
we have $z\in Z^\circ_Y(\widehat\tau)$, i.e,
\begin{equation}
\label{work1}
\begin{split}
\Tab(z,\widetilde F_\bullet)=\widehat{\P},\qquad \Tab(\widetilde x^{-1}z\widetilde x,\widetilde E_\bullet)=\Tab(z,\widetilde x\widetilde E_\bullet)=\widehat{\Q}.
\end{split}
\end{equation}
Combined with \cref{rectification}, this yields $\Tab(x y,F_\bullet)=\operatorname{Rect}(\widehat\P,q)$, see also \cref{mapEmbed}. 
A further calculation,
\begin{equation}
\label{tildeXconjugation}
\widetilde x^{-1}z\widetilde x=
\begin{pmatrix}
1  & 0\\
-x & 1
\end{pmatrix}
\begin{pmatrix}
0 &  y\\
0 & x y
\end{pmatrix}
\begin{pmatrix}
1 & 0\\
x & 1
\end{pmatrix}
=
\begin{pmatrix}
 y x &  y
\\0     & 0
\end{pmatrix},
\end{equation}
yields the claim $\Tab( y x,E_\bullet)=\Q(q)$, see \cref{springerFibre}.

The claim $ \sh\Lambda=\sh{\Lambda^+}+\sh{\Lambda^-}=\sh\Q+\sh\P$ follows from the admissibility of $\Lambda$, see \cref{LambdaAdmissible,main:admissible}.

Let $V=V_1\oplus V_2\oplus\cdots\oplus V_n$  be decomposition of $V$ as in \cref{colorDecomposition}.
Without loss of generality, we may assume that the indexing $V_i$ satisfies the following rules:
\begin{enumerate}
\item $U_k^\pm$ appears before $U_l^\pm$ if $k>l$.
\item $U_k^-$ appears before $U_k^+$.
\end{enumerate}
We construct a $\pm$-filling in the equivalence class $\Lambda$ by setting the $i^{th}$-row to be the signed Young diagram corresponding to the indecomposable subspace $V_i$.

Since $\Lambda$ is admissible, the indexing ensures that for all $i$, we have,
\begin{align*}
\Lambda^+(i)=\dim (V_i\cap V_q), && \Lambda^-(i)=\dim (V_i\cap V_p),
\end{align*}
see \cref{sydCompare,defn:admissible}. % for all $i$.
For $1\leq i\leq k$, let $\alpha_i=J(z|V_i)(1)$, the first entry of the partition $J(z|V_i)$.
The indexing ensures that the sequence $\alpha_i$ is weakly decreasing.
%Comparing with \cref{work2}, we see that $\lambda=(\alpha_1,\cdots,\alpha_n,1,1,\cdots,1)$.
Following \cref{nuLambda}, we see that $\lambda=(\alpha_1,\cdots,\alpha_n,1,1,\cdots,1)$.

Now, for each $i$, precisely one of the following holds:
\begin{align*}
\Lambda^+(i)=\Lambda^-(i)=\alpha_i-1&=l,   && \text{if }V_i\cong U_{2l}^-,\\
\Lambda^+(i)=\Lambda^-(i)+1=\alpha_i&=l+1, && \text{if }V_i\cong U_{2l+1}^+,\\
\Lambda^+(i)=\Lambda^-(i)=\alpha_i&=l,     && \text{if }V_i\cong U_{2l}^+,\\
\Lambda^+(i)+1=\Lambda^-(i)=\alpha_i&=l+1, && \text{if }V_i\cong U_{2l+1}^-.
\end{align*}
In the first two cases, the first box of the signed Young diagram corresponding to $V_i$ is \plus.
In the other cases, it is \minus.
We see that this $\pm$-filling is precisely as described in the theorem.
\end{proof}

\begin{mainTheorem}
Consider $\Lambda\in ASYD(q,p)$, $\Q\in SYD(\Lambda^+)$, and $\P\in SYD(\Lambda^-)$.
%where $\Lambda^+$ and $\Lambda^-$ are as in \cref{LambdaPM}.
Let $\lambda$ be as in \cref{nuLambda}, and let
\begin{align*}
\widehat\Q=(\Q;\lambda),&& \widehat\P=\ev{(\ev\P;\lambda)},&&\widehat\tau\overset{RSK}\longleftrightarrow(\widehat\Q,\widehat\P).
\end{align*}
Then $\tau$ is the south-west sub-matrix of $\widehat\tau$ of size $p\times q$.
\end{mainTheorem}
\begin{proof}
Let $(E_\bullet,F_\bullet,x, y)$ be a generic point in $Z_C(\tau)$, and set
$
z=
\begin{pmatrix}
0 & y\\ 
0 &x y
\end{pmatrix},
$
so that
\begin{align*}
\Tab(z,\widetilde F_\bullet)=\widehat{\P}, && \Tab(\widetilde x^{-1}z\widetilde x,\widetilde E_\bullet)=\widehat{\Q},
\end{align*}
see \cref{work1}.
Following \cref{nuLambda}, we obtain 
\begin{align}
\label{workQP}
\widehat\Q(q+1)=\sh{\widehat\Q}\lambda,&& \widehat\P(p+1)=\sh{\widehat\P}=\lambda.
\end{align}

Recall from \cref{main:bijection} that $\Q=\Tab(yx,E_\bullet)$ and $\P=\Tab(xy,F_\bullet)$.
Following \cref{tildeXconjugation}, we have $z|E_\bullet=yx$,
and hence
\begin{align*}
&&\widehat\Q(i)=\Q(i), && \forall\,1\leq i\leq q,
\end{align*}
Combined with \cref{workQP}, this yields $\widehat\Q=(\Q;\lambda)$.

Next, we see from \cref{evacuation} that
\begin{align*}
\ev{\widehat\P}=(
J(z|\nicefrac{\widetilde F_p}{\widetilde F_{p-1}});
J(z|\nicefrac{\widetilde F_p}{\widetilde F_{p-2}});
\cdots;
J(z|\nicefrac{\widetilde F_p}{\widetilde F_{1}});
J(z|\nicefrac{\widetilde F_p}{\widetilde F_{0}})
)
\end{align*}
Following \cref{work1,mapEmbed}, we have
\begin{align*}
&&J(z|\nicefrac{\widetilde F_p}{\widetilde F_{i}})=J(xy|\nicefrac{F_p}{F_i}),&&\forall\, 1\leq i\leq p,
\end{align*}
and hence,
\begin{align*}
&&\ev{\widehat\P}(i)=\ev\P(i), &&\forall\, 1\leq i\leq p.
\end{align*}
%, $\sh{\ev{\widehat\P}}=\sh{\widehat\P}$; 
%\ev{\widehat\P}(p+1)=\lambda$.
Combined with \cref{workQP}, this yields,
$
\ev{\widehat\P}=(\ev\P;\lambda).
$
Since evacuation is an involution, we obtain the desired formula, $\widehat\P=\ev{(\ev\P;\lambda)}$.

Finally, the relationship between $\widehat\tau$ and $\tau$ follows from \cref{tauHat}.
\end{proof}

\section{Projective Duality for Matrix Schubert Varieties}
\label{duality}
Consider the natural duality $SYD(q,p)\to SYD(p,q)$,
denoted $\Lambda\mapsto\Lambda^\vee$,
which simply switches the \plus\ and \minus\ boxes.
In this section,
we study the duality $(\Lambda,\Q,\P)\mapsto (\Lambda^\vee,\P,\Q)$,
and show that this recovers the projective duality on (projectivized) matrix Schubert fibres.

Let $K=GL(V_q)\times GL(V_p)$, $H=\Hom(V_q,V_p)$, $X=\Fl(V_q)\times\Fl(V_p)$, and $C=X\times H$ be as in \cref{MSV}.

\subsection{Matrix Schubert Fibres}
\label{projectiveDuality}
Fix $(E_\bullet,F_\bullet)\in X$, and let 
\begin{equation*}
B_K=\operatorname{Stab}_K(E_\bullet,F_\bullet).
\end{equation*}

For $\tau\in\mathcal{PP}(q,p)$, let $H(\tau)$ be the fibre of $C(\tau)\to X$ at the point $(E_\bullet,F_\bullet)$.
The $H(\tau)$ are precisely the $B_K$ orbits in $H$, i.e., $\nicefrac BK=\mathcal{PP}(p,q)$.
We call $H(\tau)$ a matrix Schubert fibre in $H$.

\subsection{Duality for partial permutations}
\label{ppduality}
Recall from \cref{defKH} the identification $T^*H=H\times H^\vee$.
We identify $T^*H$ as the cotangent bundle of $H^\vee$ also, via the projection map $\pi:H\times H^\vee\to H^\vee$.
Let $T^*_HH(\tau)$ be the conormal bundle of $H(\tau)$ in $H$,
and set 
\begin{align*}
\left(\overline{H(\tau)}\right)^\vee=\pi(\overline{T^*_HH(\tau)}).
\end{align*}
Now, since $\overline{T^*_HH(\tau)}$ is $B_K$-stable and irreducible, the same is true of $\left(\overline{H(\tau)}\right)^\vee$.
Consequently, $\overline{H(\tau)}^\vee$ is the closure of some matrix Schubert fibre in $H^\vee$, i.e.,
\begin{equation*}
\left(\overline{H(\tau)}\right)^\vee=\overline{H^\vee(\tau^\vee)},
\end{equation*}
for some $\tau^\vee\in\mathcal{PP}(q,p)$.
Following \cite[1.3.C]{MR2394437},
$T^*_HH(\tau)$ is also the conormal variety of $H^\vee(\tau^\vee)$ in $H^\vee$,
and $\mathbb PH^\vee(\tau^\vee)$ is in projective duality with $\mathbb PH(\tau)$.
%Consequently, the map $\tau\mapsto\tau^\vee$ induces a duality $\mathcal{PP}(p,q)\to\mathcal{PP}(q,p)$.
%Following the discussion in \cite{MR2394437},

\begin{proposition}
Let $\Lambda^\vee$ be obtained from $\Lambda$ by switching the \plus\ and \minus\ boxes.
If $\tau\in\mathcal{PP}(p,q)$ corresponds to the triple $(\Lambda,\Q,\P)$ via \cref{main:bijection},
then $\tau^\vee$ corresponds to $(\Lambda^\vee,\P,\Q)$.
\end{proposition}
\begin{proof}
Let $C^\vee=H^\vee\times X$.
Consider a generic point $(x,y,F_\bullet,G_\bullet)$ in $Z_C(\tau)$.
Following \cref{projectiveDuality,ppduality}, we have $Z_C(\tau)=Z_{C^\vee}(\tau^\vee)$,
and hence $(x,y,F_\bullet,G_\bullet)$  is generic in $Z_{C^\vee}(\tau^\vee)$.
It follows that the triple is determined again by \cref{main:PPtoTriple},
with the roles of $V_q$ and $V_p$ switched,
i.e., with $\deg(V_q)=\minus$ and $\deg(V_p)=\plus$.
This yields the claimed correspondence $\tau^\vee\leftrightarrow(\Lambda^\vee,\P,\Q)$.
%In particular, we have $\Q^\vee=\P$ and $\P^\vee=\Q$;
%and $\Lambda^\vee$ is obtained from $\Lambda$ by switching the \plus\ and \minus\ boxes.
\end{proof}

\bibliography{biblio}

\providecommand{\bysame}{\leavevmode\hbox to3em{\hrulefill}\thinspace}
\providecommand{\MR}{\relax\ifhmode\unskip\space\fi MR }
\providecommand{\MRhref}[2]{%
  \href{http://www.ams.org/mathscinet-getitem?mr=#1}{#2}
}
\providecommand{\href}[2]{#2}
\begin{thebibliography}{GKZ08}

\bibitem[BB05]{MR2133266}
Anders Bj\"orner and Francesco Brenti, \emph{Combinatorics of {C}oxeter
  groups}, Graduate Texts in Mathematics, vol. 231, Springer, New York, 2005.
  \MR{2133266}

\bibitem[CG97]{MR1433132}
Neil Chriss and Victor Ginzburg, \emph{Representation theory and complex
  geometry}, Birkh\"auser Boston, Inc., Boston, MA, 1997. \MR{1433132}

\bibitem[FN]{nishiyama}
L.~Fresse and K.~Nishiyama, \emph{A generalization of steinberg theory and an
  exotic moment map}, arXiv:1904.13156.

\bibitem[Ful92]{MR1154177}
William Fulton, \emph{Flags, {S}chubert polynomials, degeneracy loci, and
  determinantal formulas}, Duke Math. J. \textbf{65} (1992), no.~3, 381--420.
  \MR{1154177}

\bibitem[Ful97]{MR1464693}
\bysame, \emph{Young tableaux}, London Mathematical Society Student Texts,
  vol.~35, Cambridge University Press, Cambridge, 1997, With applications to
  representation theory and geometry. \MR{1464693}

\bibitem[GKZ08]{MR2394437}
I.~M. Gelfand, M.~M. Kapranov, and A.~V. Zelevinsky, \emph{Discriminants,
  resultants and multidimensional determinants}, Modern Birkh\"{a}user
  Classics, Birkh\"{a}user Boston, Inc., Boston, MA, 2008, Reprint of the 1994
  edition. \MR{2394437}

\bibitem[HT12]{MR2966826}
Anthony Henderson and Peter~E. Trapa, \emph{The exotic {R}obinson-{S}chensted
  correspondence}, J. Algebra \textbf{370} (2012), 32--45. \MR{2966826}

\bibitem[Joh10]{MR2941520}
Casey~Patrick Johnson, \emph{Enhanced nilpotent representations of a cyclic
  quiver}, ProQuest LLC, Ann Arbor, MI, 2010, Thesis (Ph.D.)--The University of
  Utah. \MR{2941520}

\bibitem[Kem82]{MR666395}
Gisela Kempken, \emph{Eine {D}arstellung des {K}\"{o}chers {$\tilde A_k$}},
  Bonner Mathematische Schriften [Bonn Mathematical Publications], 137,
  Universit\"{a}t Bonn, Mathematisches Institut, Bonn, 1982, Dissertation,
  Rheinische Friedrich-Wilhelms-Universit\"{a}t, Bonn, 1981. \MR{666395}

\bibitem[Kir16]{MR3526103}
Alexander Kirillov, Jr., \emph{Quiver representations and quiver varieties},
  Graduate Studies in Mathematics, vol. 174, American Mathematical Society,
  Providence, RI, 2016. \MR{3526103}

\bibitem[Knu70]{MR272654}
Donald~E. Knuth, \emph{Permutations, matrices, and generalized {Y}oung
  tableaux}, Pacific J. Math. \textbf{34} (1970), 709--727. \MR{272654}

\bibitem[Mac79]{MR553598}
I.~G. Macdonald, \emph{Symmetric functions and {H}all polynomials}, The
  Clarendon Press, Oxford University Press, New York, 1979, Oxford Mathematical
  Monographs. \MR{553598}

\bibitem[Rob38]{MR1507943}
G.~de~B. Robinson, \emph{On the {R}epresentations of the {S}ymmetric {G}roup},
  Amer. J. Math. \textbf{60} (1938), no.~3, 745--760. \MR{1507943}

\bibitem[Ros12]{MR2979579}
Daniele Rosso, \emph{Classic and mirabolic {R}obinson-{S}chensted-{K}nuth
  correspondence for partial flags}, Canad. J. Math. \textbf{64} (2012), no.~5,
  1090--1121. \MR{2979579}

\bibitem[Spa82]{MR672610}
Nicolas Spaltenstein, \emph{Classes unipotentes et sous-groupes de {B}orel},
  Lecture Notes in Mathematics, vol. 946, Springer-Verlag, Berlin-New York,
  1982. \MR{672610}

\bibitem[Spr69]{MR0263830}
T.~A. Springer, \emph{The unipotent variety of a semi-simple group}, Algebraic
  {G}eometry ({I}nternat. {C}olloq., {T}ata {I}nst. {F}und. {R}es., {B}ombay,
  1968), Oxford Univ. Press, London, 1969, pp.~373--391. \MR{0263830}

\bibitem[Ste88]{MR929778}
Robert Steinberg, \emph{An occurrence of the {R}obinson-{S}chensted
  correspondence}, J. Algebra \textbf{113} (1988), no.~2, 523--528. \MR{929778}

\bibitem[Tra05]{MR2128023}
Peter~E. Trapa, \emph{Symplectic and orthogonal {R}obinson-{S}chensted
  algorithms}, J. Algebra \textbf{286} (2005), no.~2, 386--404. \MR{2128023}

\bibitem[vL00]{MR1739585}
Marc A.~A. van Leeuwen, \emph{Flag varieties and interpretations of {Y}oung
  tableau algorithms}, J. Algebra \textbf{224} (2000), no.~2, 397--426.
  \MR{1739585}

\bibitem[Vog91]{MR1168491}
David~A. Vogan, Jr., \emph{Associated varieties and unipotent representations},
  Harmonic analysis on reductive groups ({B}runswick, {ME}, 1989), Progr.
  Math., vol. 101, Birkh\"{a}user Boston, Boston, MA, 1991, pp.~315--388.
  \MR{1168491}

\end{thebibliography}
\bibliographystyle{amsalpha}
\end{document}